\documentclass[11pt]{article}
\usepackage{amssymb}
\usepackage{color}
\usepackage{amsmath,amsthm}
\usepackage{graphicx}
\usepackage{empheq}
\usepackage{indentfirst}
\usepackage[hypertex]{hyperref}
 \newtheorem{thm}{Theorem}[section]

 \newtheorem{prop}{Proposition}[section]
 \theoremstyle{definition}
 \newtheorem{defn}{Definition}[section]
 \newtheorem{rem}{Remark}[section]
 \numberwithin{equation}{section}
\def\f{\frac}
\def\pa{\partial}
\def\pd #1#2{\f{\pa #1}{\pa #2}}

\def\e{\eqref}
\def\la{\lambda}

\def\vep{\varepsilon}
\def\vph{\varphi}
\def\wt #1{\widetilde{#1}}
\def\ov #1{\overline{#1}}
\def\i1n{i=1,\cdots,n}
\def\j1n{j=1,\cdots,n}
\def\ij1n{i,j=1,\cdots,n}

\def\R{\mathbb R}
\def\N{\mathbb N}

\title{Exact boundary controllability for 1-D quasilinear hyperbolic systems with a vanishing characteristic speed}
\author{
Jean-Michel Coron\thanks{
Institut universitaire de France and Universit\'{e} Pierre et Marie
Curie-Paris 6, UMR 7598 Laboratoire Jacques-Louis Lions, Paris,
F-75005 France. E-mail: {\tt coron@ann.jussieu.fr}. JMC was
partially supported by the ``Agence Nationale de la Recherche''
(ANR), Project C-QUID.},
Oliver Glass\thanks{Universit\'{e} Pierre et Marie Curie-Paris 6, UMR 7598 Laboratoire
Jacques-Louis Lions, Paris, F-75005 France. E-mail: {\tt
glass@ann.jussieu.fr}. OG was partially supported by the by the ``Agence Nationale de la Recherche''
(ANR), Project Contr\^{o}leFlux.}
and Zhiqiang Wang\thanks{School of Mathematical Sciences, Fudan
University, Shanghai 200433, China. Universit\'{e} Pierre et Marie
Curie-Paris 6, UMR 7598 Laboratoire Jacques-Louis Lions, Paris,
F-75005 France. E-mail: {\tt wzq@fudan.edu.cn}. ZW was partially
supported by the Natural Science Foundation of China grant 10701028
and Fondation Sciences Math\'{e}matiques de Paris. }
}
\date{February 11, 2009}
\begin{document}
%
%

\maketitle
\begin{abstract}
The general theory on exact boundary controllability for general first order
quasilinear hyperbolic systems requires that the characteristic speeds of system
do not vanish. This paper deals with exact boundary controllability, when this is not the case.
Some important models are also shown as applications of the main result. The strategy uses the return method,
which allows in certain situations to recover non zero characteristic speeds.
\end{abstract}
{\bf Keywords:}\quad Quasilinear hyperbolic system, vanishing characteristic speed, exact boundary controllability, return method
{\bf 2000 MR Subject Classification:}\quad 35L50, 93B05, 93C20
\section{Introduction and main results}
The general theory on exact boundary controllability for general first order
quasilinear hyperbolic systems requires that the system has non vanishing
characteristic speeds \cite{Li,LiRao}.
Several papers have dealt with hyperbolic systems having a vanishing or an identically zero characteristic speed, under various assumptions.
For systems with identically zero characteristic speeds, a general result on exact controllability
 has been obtained by using internal controls \cite{LiYu}.
It is also possible to get in this case partial controllability by boundary controls, if some
eigenvalue of the system is equal to zero identically \cite{WangYu}.
A steady state controllability holds for some special hyperbolic models with vanishing characteristic speed as Saint-Venant equations (or shallow water equations), see Gugat \cite{Gugat}. For what concerns the system of isentropic gas dynamics (which contains the Saint-Venant model), a  more general boundary controllability result for (non constant) $BV$ solutions was obtained by the second author in \cite{Glass}. \par
In this paper, we will discuss exact boundary controllability for a general hyperbolic
system which admits a vanishing characteristic speed. \par
Consider the following first order quasilinear hyperbolic system
\begin{equation}
\label{sys} \pd ut+A(u)\pd ux=0, \ (t,x) \in [0,T] \times [0,L],
\end{equation}
where $u=(u_1,\cdots,u_n)^{tr}(t,x)$ is the state of the system in some nonempty open set $\Omega \subset \mathbb{R}^n$ and the $n \times n$ matrix $A$ belongs to $C^2(\Omega;\R^{n\times n})$. \par
Let $u^*\in \Omega$ be fixed. Assume that $A(u^*)$ has $n$
real distinct eigenvalues:
\begin{equation}
\label{eigenvalue}
\la_1(u^*)<\cdots<\la_{m-1}(u^*)<\la_m(u^*)=0<\la_{m+1}(u^*)
<\cdots<\la_n(u^*),
\end{equation}
for some $m \in \{1,\cdots,n\}$, which are the characteristic speeds at which the system propagates.
Thus in a neighborhood of the equilibrium $u=u^*$, the system is strictly hyperbolic and $A(u)$
has a complete set of left (resp. right) eigenvectors
$l_1(u),\cdots,l_n(u)$ (resp. $r_1(u),\cdots,r_n(u)$):
\begin{equation}
l_i(u)A(u)=\la_i(u)l_i(u) \quad
(\text {resp.} \ A(u)r_i(u)=\la_i(u)r_i(u)), \quad \i1n.
\end{equation}
Without loss of generality, let us assume that
\begin{equation}
\label{li-rj}
l_i(u)r_j(u)=\delta_{ij}, \quad \ij1n,
\end{equation}
where $\delta_{ij}$ is Kronecker's symbol. Reducing $\Omega$ if necessary, we assume that
\begin{equation}
\forall j \in \{1, \dots, m-1\}, \ \la_j(u)<0 \text{ and } \forall j \in \{m+1, \dots, n\}, \ \la_j(u)>0, \quad \forall u\in \Omega.
\end{equation}
Now the question is: is it  possible to realize the local exact
controllability near the equilibrium $u=u^*$ only by using boundary
controls? \par
In order to overcome the difficulty of a characteristic speed vanishing at $u^{*}$, we assume the
following hypothesis:
\begin{itemize}
\item[{\bf (H)}: \ ]  for all $\vep>0$, there exists
$\alpha=(\alpha_1,\cdots,\alpha_{m-1},\alpha_{m+1},\cdots, \alpha_n)
\in L^{\infty}(0,1; \mathbb{R}^{n-1})$ with
\begin{equation}
\|\alpha\|_{L^{\infty}(0,1; \mathbb{R}^{n-1})} \leq \vep,\end{equation}
such that the solution $z\in C^0([0,1];\mathbb{R}^n)$ of the ordinary differential equation
\begin{equation} \label{SysControleDimFinie}
 \frac{dz}{ds}=\sum_{j\neq m}\alpha_j(s)r_j(z),\quad z(0)=u^*,
\end{equation}
satisfies
\begin{equation}
\la_m(z(1))\neq 0. \end{equation}
\end{itemize}
\medskip
The main result of this paper is  the following theorem:
\begin{thm}\label{main-thm}
Let \e{eigenvalue} and \emph{(H)} be true. Then, for any $\delta >0$, there exist $T>0$ and $\nu>0$ such that, for all $\vph, \psi \in C^1([0,L]; \mathbb{R}^n)$ satisfying
\begin{equation}
\|\vph(\cdot)-u^*\|_{C^1([0,L])}\leq \nu,
  \quad \|\psi(\cdot)-u^*\|_{C^1([0,L])}\leq \nu, \end{equation}
there exists $u\in C^1([0,T]\times [0,L]; \mathbb{R}^n)$ such that
\begin{align}
 & \pd ut+A(u)\pd ux=0,\quad \forall (t,x)\in [0,T] \times [0,L],\\
 &u(0,x)=\vph(x),\quad \forall x\in [0,L],\\
 & u(T,x)=\psi(x),\quad \forall x\in [0,L],\\
 & \|u(t,\cdot)-u^*\|_{C^1([0,L])} \leq \delta, \quad\forall t\in [0,T].
 \end{align}
\end{thm}
The hypothesis (H) seems quite difficult to check. However, we have
some sufficient conditions of (H) relying on {Lie brackets}.

\begin{prop}\label{Sufficient}
The following properties are sufficient conditions for (H) to hold:
\begin{itemize}
\item[\bf (H1):\ ] there exists $j\in \{1,\cdots,n\} \setminus \{ m \}$ such that
   $ \nabla \la_m (u^*) \cdot r_j(u^*) \neq 0$,
\item[\bf (H2):\ ] there exist $j,k\in \{1,\cdots,n\}  \setminus \{ m \}$ such that
   $ \nabla \la_m (u^*)\cdot [r_j,r_k](u^*) \neq 0$,
\item[\bf (H3):\ ] $A \in C^{\infty}(\Omega;\mathbb{R}^{n\times n})$ and there exists $h \in Lie\{r_1,\cdots,r_{m-1},r_{m+1},\cdots,r_n\},$
    such that $ \nabla \la_m (u^*) \cdot h(u^*) \neq 0$,
\item[\bf (H4):\ ] $A \in C^{\infty}(\Omega;\mathbb{R}^{n\times n})$ and $\{ h(u^{*}), \ h \in Lie\{r_1,\cdots,r_{m-1},r_{m+1},\cdots,r_n\} \}= \R^{n}$
and $u^{*}$ is in the closure of $\{ u \in \Omega : \ \lambda_{m}(u) \not =0 \}$.
\end{itemize}

\end{prop}
Here $Lie\{r_1,\cdots,r_{m-1},r_{m+1},\cdots,r_n\}$ denotes the Lie algebra generated by the smooth vector fields $r_1$,\dots,$r_{m-1}$,$r_{m+1}$,\dots,$r_n$.
\begin{proof}[Proof of Proposition \ref{Sufficient}]
It is a consequence of Chow and Rashevski's connectivity Theorem (see for instance \cite[Theorem 3.19, p. 135]{Coron}) that (H4) implies (H).
Next we notice that both (H1) and (H2) clearly imply (H3). So we have left to prove that (H3) implies (H).
From (H3) we deduce that the exists a direction $b \in \R^{n}$ obtained by $p$ successive Lie brackets and such that $\nabla \lambda_{m}(u^{*}) \cdot h \not =0$.
We use \cite[Lemma 1, p. 456]{Haynes-Hermes} to deduce that there are controls $\alpha$ which are arbitrarily small in $L^{\infty}$ norm such that the corresponding solution of \eqref{SysControleDimFinie}
satisfies $z(4^p t^{1/(1+p)})=u^{*}+ tb + o(t)$ as $t \rightarrow 0$. The conclusion follows.
\end{proof}
\begin{rem}
Theorem \ref{main-thm} can be regarded as a local boundary
controllability result because one can drive any initial data $\vph$
to any desired data $\psi$ near $u=u^*$ without using any internal
controls. However, since the characteristic speed $\la_m$ may change its sign
during the control period, it is difficult to describe the exact
distribution of boundary controls. To overcome this difficulty, we consider the system without boundary conditions (which is consequently under-determined), and aim at finding the solution $u$ itself. In the conservative case (where $A(u)$ is a Jacobian matrix $Df(u)$), the solution that we determine can enter the general theory of initial-boundary problems for systems of conservation laws, see in particular Amadori \cite{Amadori} and Amadori and Colombo \cite{AmadoriColombo}.
\end{rem}
\begin{rem}
Due to the hypothesis (H), one can drive the possible vanishing
characteristic speed $\la_m$ to be nonzero after sufficiently long time by
only using boundary controls. However, if some characteristic speeds of the
system are identically zero, the approach of this paper is not valid
anymore. Is boundary controllability possible in such cases, even
for some special models? Up to our knowledge, this question
remains open.
\end{rem}
\begin{rem}
We could treat the case where $A \in C^{1}(\Omega;\R^{n\times n})$, see in particular Remark \eqref{Rem31} below.
\end{rem}
The main idea to prove Theorem \ref{main-thm} is to use a
constructive approach and the return method \cite{CoronDF}.
In our framework the method consists in constructing a trajectory $\ov u\in
C^2([0,T]\times [0,L]; \mathbb{R}^n)$ of the system \e{sys}, close to ${u}^{*}$ such
that
\begin{equation}
\ov u(0,x)=\ov u(T,x)=u^*,\quad \forall x\in [0,L], \end{equation}
and that the linearized equation around $\overline{u}$ is controllable. Note indeed that the linearized equation around $u^{*}$ is not controllable.
Based on this, we can construct a solution $u\in C^1([0,T] \times
[0,L]; \mathbb{R}^n)$ to the system \e{sys} which connects the
initial and final data (which have to be sufficiently close to $u^{*}$).

As a matter of fact, we will quite not use the linearized equation. Instead, we use an argument of perturbation
of the trajectory $\ov u$ and then reduce the original control problem to a boundary control problem without vanishing
characteristic speeds, which has been solved by Li and Rao \cite{LiRao}. In the framework of systems of conservation laws, the return method has also been used in \cite{Chapouly,Coron2,Glass,Horsin}, see also \cite{AnconaMarson}. For other applications of the return method, see \cite{Coron} and the references therein. \par
Without loss of generality, we may assume the equilibrium $u^*$ to
be 0, replacing $u$ by $u-u^*$ as the unknown in the
system \e{sys} if necessary. For the convenience of statement, we denote by $C$
various positive constants in the whole paper which may change from one line to another.

The organization of this paper is as follows: in Section 2 we
construct the special trajectory $\ov u\in C^2([0,T] \times [0,L];
\mathbb{R}^n)$ of the system \e{sys} which starts at $0$ and returns
to $0$, and such that the equation linearized around $\overline{u}$ is controllable.
Then we prove the main result, Theorem \ref{main-thm}, in
Section 3. Some important applications are shown in Section
4, including Saint-Venant equations (shallow water equations), 1-D isentropic gas dynamics
equations, 1-D full gas dynamics equations and Aw-Rascle model on traffic flow and its generalization.
Finally in Appendix A, we establish a technical result.
\section{Construction of the trajectory $\ov u$}
\begin{defn} \label{defn 2.1}
Let $j\in \{1,\cdots,n\}$ and $u^0\in \Omega$. Let $s \in
[-\varepsilon_0,\varepsilon_0] \mapsto U_j(s) \in \Omega$ be the
orbit of the eigenvector field $r_j$ starting at $u^0$ (or rarefaction curves):
\begin{equation}
\f {d U_j}{ds}=r_j(U_j),\quad U_j(0)=u^0,
\end{equation}
where $\varepsilon_0>0$ is a small constant. Let $\Phi_j(s,\cdot)$ be the
corresponding flow map when $s$ varies, i.e.,
\begin{equation}
\Phi_j(s,u^0) := U_j(s),
\quad \forall s\in [\varepsilon_0,\varepsilon_0]. \end{equation}
\end{defn}

\begin{rem} For all $ s \in [-\varepsilon_0,\varepsilon_0]$ one has $u^+=\Phi_j(s,u^-)$  $\Longleftrightarrow$ $u^-=\Phi_j(-s,u^+)$.
\end{rem}
Our first proposition concerns simple waves which one can use to modify the state in $[0,L]$. \par
\begin{prop} \label{prop 2.1}
Let $j \in \{1,\cdots,n\} \setminus \{ m \}$ and
\begin{equation} \label{T>la_j-0}
T> \f {L}{|\la_j(0)|}.
\end{equation}
There exist $C>0$ and $\vep_0>0$, such that for all $\vep \in (0,\vep_0]$,
all $u^-,u^+ \in \Omega$ satisfying
\begin{equation}
\label{u-,u+}
|u^-|,|u^+|\leq \vep \text{ and }
u^+=\Phi_j(\overline{s},u^-) \text{ for some } \overline{s} \text{ such that } |\overline{s}|\leq \vep,
\end{equation}
there exists $u\in C^2([0,T]\times \mathbb{R}; \mathbb{R}^n)$ such that
 \begin{align} \label{sys-1}
 & \pd ut+A(u)\pd ux=0,\quad \forall (t,x)\in [0,T] \times \mathbb{R},
   \\\label{ini-u-}
 &u(0,x)= u^-,\quad \forall x\in [0,L],
   \\\label{fin-u+}
 & u(T,x)=u^+,\quad \forall x\in [0,L],
   \\\label{u-C1-R}
 & \|u(t,\cdot)\|_{C^1(\mathbb{R})} \leq C \vep,
 \quad\forall t\in [0,T].
\end{align}
\end{prop}
\begin{proof}
Without loss of generality, we may assume that $j\in
\{1,\cdots,m-1\}$ (the case where $j\in \{m+1,\cdots,n\}$ can be
treated similarly by symmetry in $x$, that is, replacing $x$ by $L-x$ if necessary).

In view of \e{eigenvalue} and \e{T>la_j-0}, there exist $\vep_1>0$
and $\eta>0$ small enough such that
\begin{equation}
\label{T>la_j-eta}
   T>\max_{|u|\leq \vep_1}\f {L+\eta}{|\la_j(u)|}.
\end{equation}

Let $\vep \in(0,\vep_1]$ and $u^-,u^+ \in \Omega$ be such that
\e{u-,u+} holds. By Definition \ref{defn 2.1}, it is easy to see
that
\begin{equation}
\label{Phi-C-delta}
    |\Phi_j(s,u^-)|\leq C \vep,
    \quad \forall s\in [-|\overline{s}|,|\overline{s}|].
\end{equation}
Let $\beta\in C_0^{\infty}((0,1); \mathbb{R})$ be such that
\begin{equation}
 \int_0^1 \beta(\theta) d \theta=1.
\end{equation}
Then we let
\begin{equation}
\label{over-beta}
\ov \beta(\theta) := \f {\overline{s}}{\eta}\, \beta(\f \theta{\eta}),
\end{equation}
which gives that  $\ov \beta \in C_0^{\infty}((0,\eta); \mathbb{R})$ and
\begin{equation}
 \int_0^{\eta}\ov \beta(\theta)d \theta =\overline{s}.
\end{equation}
From the above, the ordinary differential equation
\begin{equation}
\label{y(s)}
\f {dy}{d \theta}=\ov \beta(\theta)r_j(y),\quad y(0)=u^{-},
\end{equation}
admits a unique solution $y(\cdot)=\Phi_j(\sigma(\cdot),u^{-}) \in C^2([0,\eta]; \mathbb{R}^n) $, where
\begin{equation}
\sigma(s):=\int_0^s\ov \beta(\theta)d\theta,
   \quad  \forall s\in [0,\eta].
\end{equation}
Let
\begin{equation}
\label{vph}
 \vph(x):=
 \begin{cases}
 u^-, &x \leq L,\\
 y(x-L), & L < x < L+\eta,\\
 u^+, &x \geq L+\eta.
 \end{cases}
\end{equation}
In the following, we will denote by $C^{k}(\R)$ the space of functions of class $C^{k}$ whose derivatives up to order $k$ are bounded on $\R$ (and the norm $\| \cdot \|_{C^{k}(\R)}$ is in fact the norm $\| \cdot \|_{W^{k,\infty}(\R)}$ ). \par
Then by \e{Phi-C-delta}, \e{over-beta}, \e{y(s)} and \e{vph}, we
obtain that
\begin{align} \label{vph-C0-R}
 &\|\vph\|_{C^0(\mathbb{R})}
   := \sup_{x\in \mathbb{R}} |\vph(x)|\leq C \vep,
    \\\label{vph'-C0-R}
 &\|\vph'\|_{C^0(\mathbb{R})}
   :=\sup_{x\in \mathbb{R}} |\vph'(x)|\leq C \frac{\overline{s}}{\eta}
   \leq  C\vep.
 \end{align}
Now we focus on the Cauchy problem of \e{sys-1} on $\R$ with the initial condition
\begin{equation}
\label{ini-vph}
    u(0,x)=\vph(x),\quad \forall x\in \mathbb{R}. \end{equation}
It is classical that there exists a unique $C^2$ solution to the Cauchy problem \eqref{sys-1} and \eqref{ini-vph}
in small time; see for instance \cite[p. 55]{Hormander}.
Let us prove that: for the fixed time $T>0$, if $\vep $ is sufficiently small, the
Cauchy problem \e{sys-1}, \e{ini-vph} admits a unique solution $u\in
C^2([0,T]\times \mathbb{R}; \mathbb{R}^n)$ such that
\e{ini-u-} to \e{u-C1-R} hold. \par
To show that, it suffices to obtain a uniform a priori estimate of the solution in $C^{1}$ (see  \cite[Theorem 4.2.5, p. 55]{Hormander}).
In order to obtain such an a priori estimate, we
assume that the Cauchy problem \e{sys-1}, \e{ini-vph} admits already a
solution $u\in C^2([0,T_0]\times \mathbb{R}; \mathbb{R}^n)$ for some
$T_0\in (0,T)$. \par
For any $i\in \{1,\cdots,n\}$  and any point $(t,x) \in [0,T_0]
\times \mathbb{R}$, we can define the $i-th$ characteristic curve
$\xi=\xi_i(\tau)$ passing through $(t,x)$ by
\begin{equation}
\f {d \xi}{d\tau}=\la_i(u(\tau,\xi)),\quad \xi(t)=x.
\end{equation}
%
%
Introducing
\begin{equation}
\label{vi-wi}
   v_i:=l_i(u)u,\ w_i:=l_i(u)\pd ux,\quad \i1n,
\end{equation}
i.e.,
\begin{equation}
u=\sum_i v_i r_i(u),\ \pd ux=\sum_i w_i r_i(u),
\end{equation}
we know that $v_i, w_i\ (\i1n)$ satisfy the following (see \cite[p. 47ff]{Hormander} and \cite{John}):
 \begin{align} \label{dvi-dit}
 \f {d v_i}{d_i t}&=\sum_{j,k} \beta_{ikl}(u) v_k w_l,\quad \i1n,
   \\ \label{dwi-dit}
 \f {d w_i}{d_i t}&=\sum_{k,l} \gamma_{ikl}(u) w_k w_l, \quad \i1n,
 \end{align}
where
\begin{equation}
 \f {d }{d_i t}:= \pd {}t +\la_i(u) \pd {}x
\end{equation}
denotes the derivative along the $i$-th characteristic,
and where $\beta_{ikl},\gamma_{ikl}\in C^1(\Omega; \mathbb{R}^n)$ satisfy in particular
 \begin{align}\label{gamma-ikk}
 \gamma_{ikk}(u)&=0,\quad \forall i,k\in \{1,\cdots,n\},k\neq i,
  \\\label{gamma-iii}
 \gamma_{iii}(u)&=-\nabla \la_i(u) r_i(u),\quad \i1n.
 \end{align}
By \e{gamma-ikk}-\e{gamma-iii}, \e{dwi-dit} can be written as
\begin{equation}
 \label{dwi-dit-2}
   \f {d w_i}{d_i t}=\sum_{k\neq l} \gamma_{ikl}(u) w_k w_l
     - ( \nabla \la_i(u) r_i(u) )w_i^2, \quad \i1n.
\end{equation}
Combining \e{vph}-\e{vph'-C0-R} and (2.22), noticing \e{li-rj}, we have
 \begin{align}\label{vi-C-delta}
 & |v_i(0,x)|=|l_i(\vph(x))\vph(x)|\leq C\vep,
   \quad \forall x\in \mathbb{R},\forall i\in \{1,\cdots,n\},
 \\\label{wi-0}
 & w_i(0,x)=l_i(\vph(x))\vph'(x)=0,
   \quad \forall x\in \mathbb{R}, \forall i\in \{1,\cdots,n\} \setminus \{ j \},
   \\\label{wj-C-delta}
 &|w_j(0,x)|=|l_j(\vph(x))\vph'(x)| \leq C\vep,\quad \forall x\in
 \mathbb{R}.
 \end{align}
As in the proof of \cite[Theorem 4.2.5, p. 55]{Hormander},
we first assume that
\begin{equation}
\label{vi-wi-1}
  |v_i(t,x)|\leq 1,\ |w_i(t,x)|\leq 1, \quad \forall (t,x)\in [0,T_0] \times \mathbb{R}, \forall i\in \{1,\cdots,n\}.
\end{equation}
By \e{dwi-dit-2} when $i\neq j$ and \e{wi-0} , we deduce that
\begin{equation}
\label{wi-0-T0}
   w_i(t,x)=0,\quad \forall (t,x)\in [0,T_0] \times \mathbb{R}, \
   \forall i\in \{1,\cdots,n\} \setminus \{ j \},
\end{equation}
which then reduces \e{dwi-dit-2} when $i=j$ to
\begin{equation}
 \label{dwj-dit}
    \f {d w_j}{d_j t}=- ( \nabla \la_j(u) r_j(u)) w_j^2.
\end{equation}
By comparing the norm of the solution \e{dvi-dit} and \e{dwj-dit}
to the solution of the ordinary differential equation $\dot{x}= \| \nabla \lambda_{j}.r_{j}\| x^{2}$,
noticing also \e{vi-C-delta} and \e{wj-C-delta}, we deduce
that the following estimates hold
 \begin{align}\label{vi-C-delta-T0}
 & |v_i(t,x)|\leq  C \vep,\quad \forall (t,x)\in [0,T_0] \times \mathbb{R}, \ \forall i\in \{1,\cdots,n\},
  \\\label{wj-C-delta-T0}
 &|w_j(t,x)| \leq C \vep,\quad \forall (t,x)\in [0,T_0] \times \mathbb{R}.
 \end{align}
Combining \e{wi-0-T0} and \e{vi-C-delta-T0}-\e{wj-C-delta-T0}, there
exists $\vep_0 \in (0,\vep_1]$ small enough such that for any $\vep
\in (0,\vep_0]$ the assumption \e{vi-wi-1} is indeed satisfied,
and the uniform a priori estimate
\begin{equation}
 \|u(t,\cdot)\|_{C^1(\mathbb{R})}\leq C\vep,
   \quad \forall t\in [0,T_0] \end{equation}
holds for all $T_0\in (0,T)$. This proves the existence of
the solution $u\in C^2([0,T]\times \mathbb{R}; \mathbb{R}^n)$ (see again \cite[Theorem 4.2.5, p. 55]{Hormander}).
Moreover, since \e{T>la_j-eta} implies
\begin{equation}
\label{T>la_j-eta-u+}
   T> \f{L+\eta}{|\la_j(u^+)|},\end{equation}
we derive from the fact
\begin{equation}
w_i(0,x)= 0, \quad \forall x \in [L+\eta,\infty),
  \forall i\in \{1,\cdots,n\}  \end{equation}
that
\begin{equation}
w_i(T,x)=0, \quad  \forall x\in [0,\infty),
 \forall i\in \{1,\cdots,n\}, \end{equation}
which in turn implies that
\begin{equation}
u(T,x)=const.=\lim_{x\rightarrow \infty} u(T,x)=u^+,
 \quad \forall x\in [0,\infty).   \end{equation}

This concludes the proof of Proposition \ref{prop 2.1}.
\end{proof}

\begin{rem}
In the proof above, the information travels from right to left through the boundary $x=L$. For $i \in \{ 1, \dots, m-1\}$, the information would travel from left to right through the boundary $x=0$.
\end{rem}

\begin{rem}
Because \e{sys-1} is an autonomous  system, the
conclusion of Proposition \ref{prop 2.1} on $[0,T]$ can be achieved
on $[t_0,t_0+T]$ for any $t_0 \in \mathbb{R}$ by
translation in time.
\end{rem}
The next proposition prove that one can approximate the trajectory given by \eqref{SysControleDimFinie} by a trajectory composed of simple waves. \par
\begin{prop}\label{prop 2.4}
There exist $C>0$ and $\vep_0>0$ such that the following holds.
For any $\vep\in (0,\vep_0]$, any $\alpha=(\alpha_1,\cdots,\alpha_{m-1},\alpha_{m+1},\cdots,
\alpha_n) \in L^{\infty}(0,1;\mathbb{R}^{n-1})$ satisfying
\begin{equation}
\label{alpha-vep}
\|\alpha\|_{L^{\infty}(0,1;\mathbb{R}^{n-1})}\leq \vep,
\end{equation}
we consider $z\in C^0([0,1];\mathbb{R}^n)$ the solution to the ordinary
differential equation
\begin{equation}
\label{ode-z}
\frac{dz}{ds}=\sum_{j\neq m}\alpha_j(s)r_j(z),\quad z(0)=0,
\end{equation}
Then, for any $\eta>0$, there exist $p \in \N $, $i_1,\cdots,i_p \in\{1,\cdots,n\} \setminus \{m\}$ and
$t_1,\cdots,t_p\in \mathbb{R}$ such that
\begin{gather} \label{Tpetits}
\sum_{l=1}^{p} |t_l| \leq C \vep, \\
\label{ProcheZ1}
 |z(1)-\Phi_{i_p}(t_p, \cdot)\circ\cdots \circ \Phi_{i_2}(t_2,0) \, (\Phi_{i_1}(t_1,0)) | \leq \eta.
\end{gather}
\end{prop}
Proposition \ref{prop 2.4} will be established in Appendix A.
The next proposition, which establishes the existence of the special trajectory $\overline{u}$, is the principal of this section. \par
\begin{prop}\label{prop 2.5}
Let $K$ be a compact subset of $\Omega$.
There exist $C>0$ and $\vep_0>0$ such that the following holds.
For any $\vep\in (0,\vep_0]$, there exist $T>0$ and a state $\ov u^*\in K$ satisfying
\begin{equation} \label{la-i-over-u*}
\la_i(\ov u^*)\neq 0, \quad \forall i\in \{1,\cdots,n\},
\end{equation}
there exist $p \in \N$ and times $0=\tau_{0}< \tau_{1} < \dots < \tau_{2p+1}=T$ with
\begin{equation} \label{CondLiRao}
\tau_{p+1} - \tau_{p} > \displaystyle \max_{i=1 \dots n} \frac{L}{ |\lambda_{i}(\overline{u}^{*})|},
\end{equation}
and a function $\ov u \in L^{\infty}((0,T)\times \R; \mathbb{R}^n)$ such that
\begin{gather}
\label{RegInterieure}
 \overline{u}_{|[\tau_{l-1},\tau_{l}] \times \R} \in C^{2}( [\tau_{l-1},\tau_{l}] \times \R; \mathbb{R}^n), \ \forall l \in \{ 1, \dots, 2p+1 \}, \\
\label{RegInterieure2}
\overline{u}_{|[0,T] \times [0,L]} \in C^{2}([0,T]\times [0,L]; \mathbb{R}^n), \\
\label{sys-over-u}
\pd {\ov u}t+A(\ov u)\pd {\ov u}x=0,  \text{ for } (t,x) \text{ in each } [\tau_{l-1},\tau_{l}] \times \R, \ \forall l \in \{ 1, \dots, 2p+1 \}, \\
\nonumber
\ov u(0,x)=\ov u(T,x)=0,\quad  \forall x\in [0,L],\\
\nonumber
 \ov u(t,x)=\ov u^*, \quad \forall t\in [\tau_{p},\tau_{p+1}],\forall x\in [0,L], \\
\label{over-u-C-vep}
 \|\ov u (t,\cdot)\|_{C^1(\R)} \leq C\vep, \  \text{ for } t \text{ in each } [\tau_{l-1},\tau_{l}], \ \forall l \in \{ 1, \dots, 2p+1 \}.
\end{gather}
\end{prop}
\begin{proof}
By Proposition \ref{prop 2.4} and the hypothesis (H), we can deduce that there
exist $C>0$ and $\vep_1>0$ such that for any $\vep\in (0,\vep_1]$,
one can find $p \in \N$ and $i_1,\cdots,i_p \in\{1,\cdots,n\} \setminus \{ m \}$,
$t_1,\cdots,t_p\in \mathbb{R}$ such that \eqref{Tpetits} applies and
\begin{equation} \label{la-m-non-0}
\la_m(\Phi_{i_p}(t_p, \cdot)\circ\cdots \circ \Phi_{i_2}(t_2,\cdot)\circ (\Phi_{i_1}(t_1,0))) \neq 0.
\end{equation}
And thus
\begin{equation}
\label{la-i-non-0}
 \la_i(\Phi_{i_p}(t_p, \cdot)\circ\cdots \circ \Phi_{i_2}(t_2,\cdot)\circ  (\Phi_{i_1}(t_1,0)))\neq 0,\quad \forall i \in \{1,\cdots,n\}. \end{equation}
We let
\begin{equation*}
\overline{u}^{*}:=\Phi_{i_p}(t_p, \cdot)\circ\cdots \circ \Phi_{i_2}(t_2,\cdot)\circ (\Phi_{i_1}(t_1,0)).
\end{equation*}
Now for every $l\in \{1,\cdots,p\}$, let
\begin{equation} \label{T_l}
T_l := \f L{|\la_{i_l}(0)|} +1
\end{equation}
and in addition
\begin{equation} \label{T^l}
\left\{ \begin{array}{l}
\displaystyle \tau_{l}:=\sum_{k=1}^l T_k \text{ for } l=1, \dots ,p,  \smallskip \\
\displaystyle \tau_{p+1}:= \tau_{p} + \max_{i=1 \dots n} \frac{L}{ |\lambda_{i}(\overline{u}^{*})|} +1  \smallskip \\
\displaystyle \tau_{l}:= \tau_{p+1} + \sum_{k=2p+2-l}^{p} T_{k} \text{ for } l=p+2, \dots, 2p+1.
\end{array} \right.
\end{equation}
Observe that there is a symmetry with respect to the central time interval $[\tau_{p},\tau_{p+1}]$, that is, $[\tau_{p-1},\tau_{p}]$ is symmetric of $[\tau_{p+1},\tau_{p+2}]$, etc. \par
Applying Proposition \ref{prop 2.1} and Remark 2.2 with
\begin{equation*}
u_{-}=\Phi_{i_{l-1}}(t_{l-1}, \cdot)\circ\cdots \circ \Phi_{i_2}(t_2,\cdot)\circ (\Phi_{i_1}(t_1,0)) \text{ and } u_{+}=\Phi_{i_{l}}(t_{l}, \cdot)\circ\cdots \circ \Phi_{i_2}(t_2,\cdot)\circ (\Phi_{i_1}(t_1,0)),
\end{equation*}
for $l =1, \dots, p$, we deduce that provided that $\varepsilon_{0}$ is small enough, for any $\vep \in (0,\vep_0]$,
there exists $\ov u^l\in C^2([\tau_{l-1},\tau_{l}]\times \mathbb{R}; \mathbb{R}^n)$ such that
\begin{align}
&\pd {\ov u^l}t+A(\ov u^l)\pd {\ov u^l}x=0,   \quad  \forall(t,x)\in [\tau_{l-1},\tau_{l}] \times \mathbb{R},\\
 &\ov u^l(\tau_{l-1},x)=u_{-},\quad \forall x\in [0,L],\\
 & \ov u^l(\tau_{l},x)=u_{+},\quad \forall x\in [0,L],\\
 & \|\ov u^l(t,\cdot)\|_{C^1(\mathbb{R})} \leq C \vep,
 \quad\forall t\in [\tau_{l-1},\tau_{l}].
 \end{align}
Then, we let
\begin{equation}
\label{T^*} T:=\tau_{2p+1}.
\end{equation}
Finally, letting
 \begin{equation}
 \ov u(t,x):=
 \begin{cases}
  \ov u^l(t,x),  & (t,x) \in [\tau_{l-1},\tau_{l}] \times \R, l=1,\cdots,p,\\
  \ov u^*, &  (t,x)\in [\tau_{p},\tau_{p+1}] \times \R,\\
  \ov  u^{2p+1-l}(\tau_{l}-t,L-x),  & (t,x) \in [\tau_{l-1},\tau_{l}] \times \R, l=p+2,\cdots,2p+1.
 \end{cases}
 \end{equation}
we can see that $\ov u \in L^{\infty}((0,T)\times \R;\mathbb{R}^n)$
satisfies the required properties.
\end{proof}
%
%
%
%
%
%
\section{Proof of Theorem \ref{main-thm}}

In order to conclude the proof, we will use a perturbation argument together with a result by Li and Rao \cite{LiRao}.
First, we have the following perturbation result.


\begin{prop}\label{lem 3.1}
Consider $K \subset \Omega$ a nonempty compact subset.
Let $T>0$. For any $\wt u\in C^2([0,T]\times \mathbb{R}; K)$ satisfying
\begin{gather}
\label{sys-wt-u}
\pd {\wt u}t+A(\wt u)\pd {\wt u}x=0, \quad \forall (t,x)\in [0,T] \times \mathbb{R}, \\
\label{sys-wt-u0}
{\wt u}(0,x) = \tilde{\psi}(x)\quad  \forall x\in  \mathbb{R},
\end{gather}
there exist $\nu_{0}>0$ and $C>0$ such that for any $\nu \in (0,\nu_{0})$ and  any $\psi\in C^1(\R; \Omega)$ satisfying
\begin{equation}
\label{psi-u}
\|\psi(\cdot)-\tilde{u}(0,\cdot) \|_{C^1(\R)} \leq \nu,
\end{equation}
then the unique maximal solution $u \in C^1([0,T_{0}]\times \mathbb{R}; \Omega)$ of
\begin{gather}
\label{defu1}
\pd { u}t+A( u)\pd { u}x=0, \quad \forall (t,x)\in [0,T] \times \mathbb{R}, \\
\label{defu2}
 u(0,x)=\psi(x),\quad \forall x\in \R,
\end{gather}
is defined on $[0,T] \times \R$ and satisfies
\begin{gather}
\label{wt-u-C1-R}
\|u(t,\cdot) - \wt{u}(t,\cdot) \|_{C^1(\mathbb{R})} \leq C\nu, \quad\forall t\in [0,T].
\end{gather}
\end{prop}
\begin{proof}
%
%
Given $\psi\in C^1(\R; \Omega)$, there exists a local in time solution $u \in C^{1}([0,T_{0}] \times \R)$ of \eqref{defu1}-\eqref{defu2}. We show in the same time that $u$ does not blow up before $T$ and that \eqref{wt-u-C1-R} holds. \par
For that, let us make the difference of \e{sys-wt-u} and \e{defu1}, we get
\begin{gather}
\label{EqDiff1}
\pd {}t ( u-\wt u) +A(u)\pd {}x ( u-\wt u)=  (A(\wt u)-A( u)) \pd {\wt u}x,  \quad \forall (t,x) \in [0,T] \times \mathbb{R},\\
\label{EqDiff2}
u(0,x)- \wt u(0,x)= {\psi}(x)-\wt \psi(x),\quad \forall  x\in \mathbb{R}.
\end{gather}
By Gronwall's inequality we deduce that
\begin{equation}
\label{wt-u-over-u}
   \| u(t,\cdot)-\wt u(t,\cdot)\|_{C^0(\mathbb{R})} \leq
   C \|\psi-\wt \psi\|_{C^0(\mathbb{R})} \leq C\nu.
   \quad \forall t\in [0,T].
\end{equation}
Differentiating  \eqref{EqDiff1} with respect to $x$ and observing that $\wt u$ is of class $C^{2}$, we can use the same Gronwall argument to infer \eqref{wt-u-C1-R} and that the maximal solution is defined on $[0,T]$. \par
\end{proof}

\begin{rem} \label{Rem31}
We could use only a $C^{1}$ regularity assumption on $\tilde{u}$ provided that this $\tilde{u}$ has the particular structure given by Proposition \ref{prop 2.1}.
While the estimate \eqref{wt-u-over-u} should be replaced by a weaken one (but sufficient for the proof of Theorem \ref{main-thm}):
\begin{equation}
\|u(T,\cdot)-\wt {u}(T,\cdot)\|_{C^1([0,L])} \leq C\nu.
\end{equation}

\end{rem}

\begin{rem} \label{RemInvarTrans}
As previously, the
conclusion of Lemma \ref{lem 3.1} on $[0,T]$ can be achieved on
$[t_0,t_0+T]$ for any $t_0 \in \mathbb{R}$ by
translation in time.
\end{rem}

\noindent {\bf Proof of Theorem \ref{main-thm}:} Again, we may assume the equilibrium $u^*$ to be 0, otherwise
we can replace $u$ by $u-u^*$ as the unknown in the system \e{sys}.

By Proposition \ref{prop 2.5}, we can deduce
that: there exist $C>0$, $\vep_0>0$ and $T>0$, such that for any $\vep\in
(0,\vep_0]$, there exists $\ov u \in L^{\infty}((0,T) \times \R; \mathbb{R}^n)$ such
that \e{la-i-over-u*}-\e{over-u-C-vep} hold.

For every $l\in \{1,\cdots,p\}$, let  $\tau_{l}$ be given by
\e{T^l}. Let
\begin{equation}
u^0(0,x)=\vph(x),\quad \forall x\in [0,L]. \end{equation}
The proof relies on a induction argument on $l$. By Proposition \ref{lem 3.1}, we see that there exist
$C>0$, $\vep_{_l}>0$ and $\nu_{_l}>0$, for any $\vep\in
(0,\vep_{_l}]$ and any $\nu\in (0,\nu_{_l}]$, if
\begin{equation}
\|u^{l-1}(\tau_{l-1},\cdot)-\ov u(\tau_{l-1},\cdot)\|_{C^1([0,L])}\leq \nu,
  \end{equation}
then there exists $u^l\in C^1([\tau_{l-1},\tau_l]\times \mathbb{R};
\mathbb{R}^n)$ such that
 \begin{align}\label{}
 &\pd {u^l}t+A(u^l)\pd {u^l}x=0,
   \quad \forall (t,x) \in [\tau_{l-1},\tau_l] \times \mathbb{R},
   \\\label{}
 & u^l(\tau_{l-1},x)=u^{l-1}(\tau_{l-1},x),\quad \forall x\in [0,L],
   \\\label{}
 & \| u^{l}(t,\cdot)\|_{C^1(\mathbb{R})} \leq C \vep+ C\nu,
   \quad\forall t\in [\tau_{l-1},\tau_l],
   \\\label{}
 & \|u^{l}(\tau_l,\cdot)-\ov u(\tau_l,\cdot)\|_{C^1([0,L])}\leq C \nu.
 \end{align}
Therefore, there exist $C>0$, $\vep_{\mathbf{f}}>0$ and
$\nu_{\mathbf{f}}>0$, such that for any $\vep\in (0,\vep_{\mathbf{f}}]$ and
for any $\nu\in (0,\nu_{\mathbf{f}}]$, if
\begin{equation}
\|\vph\|_{C^1([0,L])}\leq \nu,\end{equation}
then there exists $u^{\mathbf{f}}\in
C^1([0,\tau_{p}]\times [0,L]; \mathbb{R}^n)$ such that
 \begin{align}\label{}
 &\pd {u^{\mathbf{f}}}t+A(u^{\mathbf{f}})\pd {u^{\mathbf{f}}}x=0,
   \quad \forall (t,x)\in [0,\tau_{p}]\times [0,L],
   \\\label{}
 & u^{\mathbf{f}}(0,x)=\vph(x),\quad \forall x\in [0,L],
   \\\label{}
 & \|u^{\mathbf{f}}(t,\cdot)\|_{C^1([0,L])} \leq C \vep+ C\nu,
   \quad\forall t\in [0,\tau_{p}],
   \\\label{}
 & \|u^{\mathbf{f}}(\tau_{p},\cdot)-\ov u(\tau_{p},\cdot)\|_{C^1([0,L])}\leq C \nu.
 \end{align}
In the same way and in view of Remark \ref{RemInvarTrans}, there exist $C>0$,
$\vep_{\mathbf{b}}>0$ and $\nu_{\mathbf{b}}>0$, such that for any $\vep\in
(0,\vep_{\mathbf{b}}]$ and for any $\nu\in (0,\nu_{\mathbf{b}}]$, if
\begin{equation}
\|\psi\|_{C^1([0,L])}\leq \nu,
\end{equation}
then there exists $u^{\mathbf{b}}\in C^1([\tau_{p+1},T]\times [0,L];
\mathbb{R}^n)$ such that
 \begin{align}\label{}
 &\pd {u^{\mathbf{b}}}t+A(u^{\mathbf{b}})\pd {u^{\mathbf{b}}}x=0,
   \quad \forall (t,x)\in [\tau_{p+1},T] \times [0,L],
   \\\label{}
 & u^{\mathbf{b}}(T,x)=\psi(x),\quad \forall x\in [0,L],
   \\\label{}
 & \|u^{\mathbf{b}}(t,\cdot)\|_{C^1([0,L])} \leq C \vep+ C\nu,
   \quad\forall t\in [\tau_{p+1},T],
   \\\label{}
 & \|u^{\mathbf{b}}(\tau_{p+1},\cdot)-\ov u(\tau_{p+1},\cdot)\|_{C^1([0,L])}\leq C \nu.
 \end{align}

Now we can apply the result of Li and Rao
\cite{LiRao} near the equilibrium of $\ov u(\tau_{p},\cdot) = \overline{u}(\tau_{p+1},\cdot) = \overline{u}^{*}\in \Omega$: due to \eqref{CondLiRao}
there exists $\nu_{\mathbf{m}}>0$, such that for any $\nu\in
(0,\nu_{\mathbf{m}}]$, if $\|u(\tau_{p},\cdot) - \overline{u}^{*}\|_{C^{1}([0,L])}$ and $\|u(\tau_{p+1},\cdot)  - \overline{u}^{*}\|_{C^{1}([0,L])}$ are small enough, there exists $u^{\mathbf{m}}\in C^1([\tau_{p},\tau_{p+1}]\times [0,L]; \mathbb{R}^n)$ such that
 \begin{align}\label{}
 &\pd {u^{\mathbf{m}}}t+A(u^{\mathbf{m}})\pd {u^{\mathbf{m}}}x=0,
   \quad \forall (t,x)\in [\tau_{p},\tau_{p+1}] \times [0,L],
   \\\label{}
 & u^{\mathbf{m}}(\tau_{p},x)=u^{\mathbf{f}}(\tau_{p},x),\quad \forall x\in [0,L],
   \\\label{}
 & u^{\mathbf{m}}(\tau_{p+1},x)=u^{\mathbf{b}}(\tau_{p+1},x),\quad \forall x\in [0,L],
   \\\label{}
 & \|u^{\mathbf{m}}(t,\cdot)\|_{C^1([0,L])} \leq C\nu,
   \quad\forall t\in [\tau_{p},\tau_{p+1}].
 \end{align}

Combining all of the above, there exists $C>0$ such that
for any $\delta>0$, there exist $\vep>0$ and $\nu>0$ small enough,
such that for any $\vph,\psi \in C^1([0,L]; \mathbb{R}^n)$ satisfying
\begin{equation}
\|\vph\|_{C^1([0,L])}\leq \nu,
  \quad \|\psi\|_{C^1([0,L])}\leq \nu,
\end{equation}
one can construct $u\in C^1([0,T]\times [0,L]; \mathbb{R}^n)$ by
\begin{equation}
 u(t,x)=
 \begin{cases}
 u^{\mathbf{f}}(t,x), & \forall (t,x) \in [0,\tau_{p}] \times [0,L],\\
 u^{\mathbf{m}}(t,x), & \forall (t,x) \in  [\tau_{p},\tau_{p+1}] \times [0,L],\\
 u^{\mathbf{b}}(t,x), & \forall (t,x) \in  [\tau_{p+1},T] \times [0,L].
 \end{cases}
 \end{equation}
Now this function $u$ clearly satisfies
 \begin{align}
 & \pd ut+A(u)\pd ux=0,\quad \forall (t,x)\in [0,T] \times [0,L],\\
 &u(0,x)=\vph(x),\quad \forall x\in [0,L],\\
 & u(T,x)=\psi(x),\quad \forall x\in [0,L],\\
 & \|u(t,\cdot)\|_{C^1([0,L])} \leq C\vep+C\nu\leq \delta,
   \quad\forall t\in [0,T].
 \end{align}
This finishes the proof of Theorem \ref{main-thm}.
%
%
%
%
\section{Some models}
\noindent
{\bf Model 1:} Saint-Venant equations (shallow water equations)
\cite{Coron2,Gugat,GugatLeugering,Halleux}:
\begin{equation}\label{Saint-Venant}
 \begin{split}
 &\pd {H}{t}+\pd {}{x}(HV)=0,\\
 &\pd {V}{t} +\pd {}{x} (\f{V^2}2+gH)=0,\\
 \end{split}
 \end{equation}
where $g>0$ is the gravity constant.
 Let $U=(H,V)^{tr}$, \e{isentropic} is reduced to
\begin{equation}
     U_t+A(U)U_x=0
 \end{equation}
with
\begin{equation}
 A(U)=\left(
        \begin{array}{cc}
          V & H \\
          g & V \\
        \end{array}
      \right).
 \end{equation}
By the study of Model 2 (see below), Theorem \ref{main-thm} can be applied
to \e{Saint-Venant} near the equilibrium $U^*:=(H^*,V^*)$ where
$V^*=\sqrt{gH^*}$ with $H^*>0$ or near the equilibrium $U^{\star}:=
(H^{\star},V^{\star})$ where $V^{\star}=-\sqrt{gH^{\star}}$ with
$H^{\star} >0$. \par
\medskip
\noindent
{\bf Model 2:} 1-D isentropic gas dynamics equations in Eulerian
coordinates \cite{Glass}:
\begin{equation}\label{isentropic}
 \begin{split}
 &\pd {\rho}{t}+\pd {m}{x}=0,\\
 &\pd {m}{t} +\pd {}{x} (\f {m^2}{\rho}+p)=0,\\
 \end{split}
 \end{equation}
where
\begin{equation}
p=K \rho^{\gamma}\quad (K>0,\ 1 < \gamma <3).
\end{equation}
We can see \e{Saint-Venant} is a special case of \e{isentropic} when $p=g \rho^2/2$. \par
Moreover, let $U=(\rho,u)^{tr}$, \e{isentropic} is reduced to
\begin{equation}
 U_t+A(U)U_x=0
 \end{equation}
with
\begin{equation}
 A(U)=\left(
        \begin{array}{cc}
          u & \rho \\
          \f {p'(\rho)}{\rho} & u \\
        \end{array}
      \right).
 \end{equation}
The characteristic speeds and the corresponding eigenvectors are
 \begin{align}
 &\la_1(U)=u-\sqrt{p'(\rho)},\quad \la_2(U)=u+\sqrt{p'(\rho)},\\
 & r_1(U)=(\f {\rho}{\sqrt{p'(\rho)}},-1)^{tr},
 \quad r_2(U)=(\f {\rho}{\sqrt{p'(\rho)}},1)^{tr}.
  \end{align}

Let $U^*:=({\rho}^*,u^*)$ where  $u^*=\sqrt{p'(\rho ^*)}$ with
${\rho}^* >0$, that is, the fluid reaches the sound speed.
Then it is easy to check that
\begin{equation}
\la_1(U^*)=0<\la_2(U^*)=2\sqrt{p'(\rho ^*)} \end{equation}
and the hypothesis (H1) is satisfied as:
\begin{equation}
\nabla \la_1 (U^*)  \cdot  r_2(U^*)
 =\f {3-\gamma}2 > 0. \end{equation}
Similarly, if we let $U^{\star}:=({\rho}^{\star},u^{\star})$ where
$u^{\star}=\sqrt{p'(\rho ^{\star})}$ with ${\rho}^{\star} >0$ (which is the symmetric case of the latter), one can see that
\begin{equation}
\la_1(U^{\star})=-2\sqrt{p'(\rho ^{\star})}<\la_2(U^{\star})=0 \end{equation}
and the hypothesis (H1) is satisfied as:
\begin{equation}
\nabla \la_2 (U^{\star})  \cdot  r_1(U^{\star})
 =\f {\gamma-3}2 < 0. \end{equation}
Therefore, Theorem \ref{main-thm} can be applied to \e{isentropic}
near the equilibrium $U^*$ or $U^{\star}$. \par
\medskip
\noindent
{\bf Model 3:} 1-D full gas dynamics equations in Eulerian coordinates
\cite{Smoller}:
\begin{equation}
\label{1d-gas}
\begin{split}
&\pd {\rho}{t}+\pd {}{x} (\rho u)=0,\\
&\pd {}{t} (\rho u)+\pd {}{x} (\rho u^2+p)=0,\\
&\pd {}{t} \Big[\rho (\f {u^2}2 +e)\Big] +\pd {}{x} \Big[\rho u (\f
  {u^2}2+e)+pu\Big]=0.
\end{split}
\end{equation}
Assume the gas is polytropic, so that
\begin{equation}
e=c_v T =\f {c_vR\rho}{p} \quad (c_v>0,\ R>0)\end{equation}
and
\begin{equation}
p=k \mathrm{e}^{\f S{c_{_v}}} \rho^{\gamma}
   \quad (k>0,\ 1 < \gamma <3). \end{equation}
Thus, on the domain of $\rho>0$, we have $p_{\rho}>0$,
$p_{\rho\rho}>0$ and $p_{_S}>0$.
Model 3 generalizes Model 2 if we let $m:=\rho u$ and $S\equiv S_0\in\mathbb{R}$. \par
Let $U=(\rho,u,S)^{tr}$, then \e{1d-gas} can be rewritten as
\begin{equation}
U_t+A(U)U_x=0,\end{equation}
with
\begin{equation}
A(U)=\left(
            \begin{array}{ccc}
              u & \rho & 0 \\
              \f {p_{\rho}}{\rho} & u & \f {p_{_S}}{\rho} \\
              0 & 0 & u \\
             \end{array}
          \right).
   \end{equation}
The characteristic speeds and the corresponding eigenvectors are
\begin{align}
 &\la_1(U)=u-c,\quad \la_2(U)=u,\quad \la_3(U)=u+c, \\
 & r_1(U)=(\rho,-c,0)^{tr}, \quad r_2(U)=(p_{_S},0,-p_{\rho})^{tr},
 \quad r_3(U)=(\rho,c,0)^{tr},
\end{align}
with $c=\sqrt{p_{\rho}}$.

Let $U^*:=(\rho^*,0,S^*)$ where $\rho^* >0, S^*\in \mathbb{R}$, then
it is easy to check that
\begin{equation}
\la_1(U^*)<\la_2(U^*)=0<\la_3(U^*) \end{equation}
and the hypothesis (H1) is satisfied as:
\begin{equation}
\nabla \la_2 (U^*) \cdot r_1(U^*)=-c(U^*)<0  \quad \text{or} \quad \nabla \la_2 (U^*) \cdot r_3(U^*)=c(U^*)>0.
\end{equation}
Therefore, we can apply Theorem \ref{main-thm} to obtain boundary
controllability for \e{1d-gas} near the equilibrium $U^*$. \par
\medskip
\noindent
{\bf Model 4:} AR and MAR traffic flow system \cite{AW,MAR}:
\begin{equation}\label{traffic-flow}
 \begin{split}
 &\pd {\rho}{t}+\pd {}{x} (\rho u)=0,\\
 &\pd {}{t} (\rho(u+p(\rho))) +\pd {}{x} (\rho u (u+p(\rho)))=0,
 \end{split}
 \end{equation}
with
\begin{equation} \tag{AR} \label{AR}
    p= \rho^{\gamma}\quad (\gamma>0) ,
\end{equation}
and
\begin{equation} \tag{MAR} \label{MAR}
    p= (\frac{1}{\rho} -\frac{1}{\rho_{0}})^{-\gamma} \quad (\gamma>0,\ \rho_0>0) .
\end{equation}
We deduce system \eqref{AR} form system \eqref{MAR} by letting $\rho_{0} = +\infty$. \par
Let $U=(\rho,u)^{tr}$, \e{traffic-flow} is reduced to
\begin{equation}
U_t+A(U)U_x=0,
\end{equation}
with
\begin{equation}
A(U)=\left( \begin{array}{cc}
   u & \rho \\
   0 & u-\rho p'(\rho)
\end{array} \right).
 \end{equation}
The characteristic speeds and the corresponding eigenvectors are
\begin{align}
 &\la_1(U)=u-\rho p'(\rho),\quad \la_2(U)=u,\\
 & r_1(U)=(1, -p'(\rho))^{tr}, \quad r_2(U)=(1, 0)^{tr}.
\end{align}
Let $U^*:=({\rho}^*,u^*)$ where  $u^*=\rho^* p'(\rho ^*)$ with
$0<{\rho}^*< \rho_{0}$, then we have
\begin{equation}
\la_1(U^*)=0<\la_2(U^*)=\rho^* p'(\rho ^*)>0 ,
\end{equation}
and the hypothesis (H1) is satisfied as:
\begin{equation}
\nabla \la_1 (U^*)  \cdot r_2(U^*) =-p'(\rho^*)-\rho^* p''(\rho^*)=- \gamma (\rho^{*})^{-3} ( \frac{1}{\rho^{*}} - \frac{1}{\rho_{0}})^{-\gamma-2} (\frac{\rho^{*}}{\rho_{0}}+\gamma) <0.
\end{equation}
Similarly, if we let $U^{\star}:=({\rho}^{\star},0)$ with
$0<{\rho}^{\star}< \rho_{0}$, then it is easy to check that
\begin{equation}
\la_1(U^{\star})=-\rho^{\star} p'(\rho^{\star})<\la_2(U^{\star})=0,
\end{equation}
and the hypothesis (H1) is satisfied as:
\begin{equation} \nabla \la_2 (U^{\star}) \cdot  r_1(U^{\star})
 =-p'(\rho^{\star})=-\gamma (\rho^{\star})^{-2} ( \frac{1}{\rho^{\star}} - \frac{1}{\rho_{0}})^{-\gamma-1} < 0.
\end{equation}
Theorem \ref{main-thm} can thus be applied to \e{traffic-flow}
near the equilibrium $U^*$ or $U^{\star}$. \par
%
%
%
%
%
%
\appendix
\section{Proof of Proposition \ref{prop 2.4}}
\label{Proof2.4}
Proposition \ref{prop 2.4} belongs to the folklore of finite-dimensional control theory (see in particular Fillipov \cite{Filippov}). Since we have not found the exact required formulation in the literature, we give the proof in details for the sake of completeness. \par
We begin with a few notations.
\begin{defn}\label{defn 2.3}
$\mathcal{P}_{(a,b)}^N \subset L^{\infty}(a,b; \mathbb{R}^N)$ is
defined as the set consisting of all piecewise constant vector
functions on $(a,b)$. Next
$\mathcal{F}_{(a,b)}^N \subset \mathcal{P}_{(a,b)}^N$ is defined as
the set consisting of all piecewise constant vector functions on
$(a,b)$ with at most one nontrivial component, i.e.,
$f=(f_1,\cdots,f_N)^{tr}\in \mathcal{F}_{(a,b)}^N$ if and only if there exist $p \in \N$,
indices $i_1,\cdots,i_p \in \{1,\cdots,N\}$, constants $f_{i_1}^1,\cdots,f_{i_p}^p\in \mathbb{R}$
and $a=t_0<t_1<\cdots<t_p=b$ such that
\begin{equation}
f(t)=f_{i_l}^l e_{i_l},\quad  \forall t\in (t_{l-1},t_l),
 l=1,\cdots,p,
\end{equation}
where $e_1,\cdots,e_N$ denote the standard basis of $\mathbb{R}^N$.
\end{defn}
Now we deduce the following statement.
\begin{prop}\label{prop 2.3}
$\mathcal{F}_{(0,1)}^N$ is dense in $L^{\infty}(0,1;\mathbb{R}^N)$
with respect to the weak-$*$ topology, more precisely,  for any $f\in
L^{\infty}(0,1;\mathbb{R}^N)$, there exists a sequence
$\{f^k\}_{k=1}^{\infty} \subset \mathcal{F}_{(0,1)}^N$ such that
\begin{equation}
\label{fk-weak*-limit}
\lim_{k\rightarrow \infty}\int_0^1 f^k(t)\cdot h(t) dt = \int_0^1 f(t)\cdot h(t) dt,
\quad \forall h\in  L^1(0,1;\mathbb{R}^N),
\end{equation}
\begin{equation}
\label{fk-C-f}
\|f^k\|_{L^{\infty}(0,1; \mathbb{R}^N)} \leq C \|f\|_{L^{\infty}(0,1; \mathbb{R}^N)}, \quad  \forall k\in \mathbb{N}.
\end{equation}
\end{prop}
\begin{proof}
It is classical that $\mathcal{P}_{(0,1)}^N$ is dense in $L^{\infty}(0,1;\mathbb{R}^N)$ for the weak-$*$ topology (moreover one can require \eqref{fk-C-f} to hold on an approximating sequence). Hence it suffices to prove that
$\mathcal{F}_{(0,1)}^N$ is dense in $\mathcal{P}_{(0,1)}^N$ with
respect to weak-$*$ topology.
To do this, we first prove \e{fk-weak*-limit} in the simpler case
where $f$ is a constant function:
\begin{equation}
f(t)=\ov f=(\ov f_1,\cdots,\ov f_N)^{tr} \in \mathbb{R}^N,
   \quad \forall t\in [0,1].
\end{equation}
For any $k \in \N$, we let $\ov f^k \in
\mathcal{F}_{(0,1)}^N$ be defined as
\begin{equation}
\ov f^k(t):=N \ov f_i e_i,
 \ \forall t\in \Big(\f {(j-1)N+i-1}{kN},
 \f {(j-1)N+i}{kN}\Big), \forall j\in \{1,\cdots,k\},
 i=1,\cdots N.\end{equation}
Clearly $\ov f^k$ converges weakly-$*$ to $f$ in $L^{\infty}(0,1;\mathbb{R}^N)$ as $k$ tends to $\infty$.
Now we treat the general case where $f\in \mathcal{P}_{(0,1)}^N$. We introduce times $0=t_0<t_1<\cdots<t_p=1$ such that
\begin{equation}
 f(t)=\ov f^l, \quad \forall t\in (t_{l-1},t_l), l=1,\cdots,p,
 \end{equation}
where the $\overline{f}^{l}$ are constants. \par
From the previous arguments, we can obtain by translation and
scaling that there exists $\{\ov f^{l^k}\}_{k=1}^{\infty} \subset
\mathcal{F}_{(t_{l-1},t_l)}^N$ such that $\ov f^{l^k}$ converges
weakly-$*$ to $\ov f^l$ in $L^{\infty}(t_{l-1},t_l; \mathbb{R}^N)$
as $k$ tends to $\infty$. Finally, for any $k\in \N$, we
let
\begin{equation}
f^k(t):=\ov f^{l^k}(t), \quad t\in (t_{l-1},t_l), l=1,\cdots,p.
\end{equation}
It is obvious that $\{f^k\}_{k=1}^{\infty} \subset
\mathcal{F}_{(0,1)}^N$ and  $f^k$ converges weakly-$*$ to $f$ in
$L^{\infty}(0,1; \mathbb{R}^N)$ as $k$ tends to $\infty$, i.e.,
\e{fk-weak*-limit} holds. Observe that \eqref{fk-C-f} holds.
\end{proof}
\noindent
{\bf Back to the proof of Proposition \ref{prop 2.4}.} Since $\alpha\in L^{\infty}(0,1; \mathbb{R}^{n-1})$, the solution
$z\in C^0([0,1];\mathbb{R}^n)$ to the ordinary differential equation
\e{ode-z} is Lipschitz continuous, since
\begin{equation} \label{A9}
z(s)=\int_0^s \sum_{j\neq m}
  \alpha_j(\theta)r_j(z(\theta)) d \theta,
  \quad \forall s\in [0,1]. \end{equation}
Let $\alpha$ be such that \e{alpha-vep} holds. By Proposition
\ref{prop 2.3}, there exists a sequence $\{\alpha^k
\}_{k=1}^{\infty} \subset \mathcal{F}_{(0,1)}^{n-1}$ with the
notation $ \alpha^k:=(\alpha_1^k,\cdots, \alpha_{m-1}^k,
\alpha_{m+1}^k, \cdots, \alpha_n^k)$, which converges weakly-$*$ to
$\alpha$ in $L^{\infty}(0,1; \mathbb{R}^{n-1})$ and
\begin{equation}
\label{alpha-C-vep}
  \|\alpha^k \|_{L^{\infty}(0,1; \mathbb{R}^{n-1})}
  \leq C \|\alpha\|_{L^{\infty}(0,1; \mathbb{R}^{n-1})}
  \leq C \vep, \quad \forall k\in \N. \end{equation}

 Let $z^k\in C^0([0,S_k];
\mathbb{R}^n)$ be the solution to the Cauchy problem
\begin{equation}
 \frac{dz^k}{ds}=\sum_{j\neq m} \alpha_j^k(s)r_j(z^k),
  \quad z^k(0)=0,\end{equation}
where $S_k \in (0,1]$. By \eqref{alpha-C-vep}, $z^k$ is uniformly Lipschitz continuous.
%

If $\vep$ is small enough, then by \e{alpha-vep}, we can deduce that
\begin{equation}
S_k=1,
\end{equation}
that is, $z_{k}$ is defined on the whole time interval $[0,1]$, for all $k\in \N$, and
\begin{equation}
   \|z^k\|_{W^{1,\infty}(0,1;\mathbb{R}^n)} \leq C \vep,
  \quad \forall k\in \N. \end{equation}

By the Arzel\`{a}-Ascoli Theorem, there exists a subsequence
$\{z^{k^l}\}_{l=1}^{\infty} \subset \{z^k\}_{k=1}^{\infty}$ and
$z^{\infty}\in C^0([0,1]; \mathbb{R}^n)$ such that
 $z^{k^l}$ converges to $z^{\infty}$ in $C^0([0,1]; \mathbb{R}^n)$ as
$l$ tends to $\infty$.
Now it is straightforward to pass to the limit in \eqref{A9} (even, the limit is unique). The conclusion follows. \par


\begin{thebibliography}{00}
%
\bibitem{Amadori} D. Amadori, Initial-boundary value problems for nonlinear systems of conservation laws, NoDEA Nonlinear Differential Equations Appl. 4 (1997), no. 1, pp. 1--42.
%
\bibitem{AmadoriColombo} D. Amadori, R. M. Colombo, Continuous dependence for $2\times 2$ conservation laws with boundary, J. Differential Equations 138 (1997), no. 2, pp. 229--266.
%
\bibitem{AnconaMarson} F. Ancona, A. Marson, On the attainable set for scalar nonlinear conservation laws with boundary control, SIAM J. Control Optim. 36 (1998), no. 1, pp. 290--312.
%
\bibitem{AW} A. Aw, M. Rascle,
{Resurrection of ``second order" models of traffic flow},
 SIAM J. Appl. Math. 60 (2000), pp. 916--938.
%
\bibitem{MAR} F. Berthelin, P. Degond, M. Delitala, M. Rascle, A model for the formation and evolution of traffic jams, Arch. Ration. Mech. Anal. 187 (2008), no. 2, pp. 185--220.
%
\bibitem{Chapouly} M. Chapouly, Global controllability of nonviscous Burgers type equations, C. R. Math. Acad. Sci. Paris 344 (2007), no. 4, pp. 241--246.
%
\bibitem{CoronDF} J.-M. Coron, Global Asymptotic Stabilization for controllable systems without drift, { Math. Control Signal Systems} 5 (1992), pp. 295--312.
%
\bibitem{Coron2} J.-M. Coron,
{Local controllability of a 1-D tank containing a fluid modeled by the shallow water equations},
ESAIM: Control Opt. Calc. Var. 8 (2002), pp. 513--554.
%
\bibitem{Coron}  J.-M. Coron,
\emph{Control and Nonlinearity}, Mathematical Surveys and Monographs {\bf  136},
American Mathematical Society, Providence, RI, 2007.
%
\bibitem{Filippov}  A.F. Filippov, Classical solutions of differential equations with multi-valued right-hand side, SIAM J. Control 6 (1967), pp. 609--621.
%
\bibitem{Glass}  O. Glass,
{On the controllability of the 1-D isentropic Euler equation},
J. Eur. Math. Soc. 9 (2007), pp. 427--486.
%
\bibitem{Gugat}  M. Gugat,
{Boundary controllability between sub- and supercritical flow},
SIAM J. Control Optim. 42 (2003), pp. 1056--1070.
%
\bibitem{GugatLeugering}  M. Gugat, G. Leugering, Global boundary controllability of the de St. Venant equations between steady states, Ann. Inst. H. Poincar\'{e} Anal. Non Lin\'{e}aire 20 (2003),  no. 1, pp. 1--11.
%
\bibitem{Halleux}
J. de Halleux, C. Prieur, J.-M. Coron, B. d'Andr\'{e}a-Novel, G. Bastin,
{Boundary feedback control in networks of open channels},
Automatica, 39 (2003),  pp. 1365--1373.
%
\bibitem{Haynes-Hermes} G. W. Haynes, H. Hermes, Nonlinear controllability via Lie theory, SIAM J. Control 8 (1970), pp. 450--460.
%
\bibitem{Hormander}  L. H\"{o}rmander,
\emph{Lectures on Nonlinear hyperbolic Differential Equations},
Math\'{e}matiques \& Application {\bf 26}, Springer-Verlag, Berlin, 1997.
%
\bibitem{Horsin} T. Horsin, On the controllability of the Burgers equation, ESAIM: Control Opt. Calc. Var. 3 (1998), pp. 83-95.
%
\bibitem{John} F. John,
{Formation of singularities in one-dimensional nonlinear wave propagations},
Comm. Pure Appl. Math. 27 (1974), pp. 377--405.
%
%
\bibitem{Li} T.-T. Li,
\emph{Controllability and Observability for Quasilinear Hyperbolic Systems},
Springer Verlag, New York, 2008.
%
\bibitem{LiRao} T.-T. Li, B. P. Rao,
{Exact boundary controllability for quasilinear hyperbolic systems},
SIAM J. Control Optim. 41 (2003),  pp. 1748--1755.
%
\bibitem{LiYu} T.-T. Li, L. X. Yu,
{Exact controllability for first order quasilinear hyperbolic systems
   with zero eigenvalues}, Chinese Ann. Math. Ser. B 24 (2003), pp. 415--422.

\bibitem{Smoller} J. Smoller,
\emph{Shock Waves and Reaction-Diffusion Equations}.
Grundlehren der mathematischen Wissenschaften {\bf 258},
Springer-Verlag, New York, 1983.
%
\bibitem{WangYu} Z. Q. Wang, L. X. Yu,
{Exact boundary controllability for one-dimensional adiabatic flow system
(in Chinese)}, Appl. Math. J. Chinese Univ. Ser. A. 23 (2008), pp. 35--40.
\end{thebibliography}
\end{document}